\documentclass{article}
 \usepackage[margin=1in]{geometry} 
\usepackage{amsmath,amsthm,amssymb,amsfonts, fancyhdr, comment, parskip, tikz, caption, subcaption, graphicx, xcolor}
\usepackage[colorlinks]{hyperref}
\usetikzlibrary{graphs}
\usetikzlibrary{shadings}
\usetikzlibrary{shapes.geometric}
\usetikzlibrary{arrows}
\usepackage{mathtools}
\usepackage{upgreek}
\usepackage{subcaption}
\usepackage{bm}
\usepackage{tikz-cd}
\usepackage{enumerate}

\newtheorem{theorem}{Theorem}

\newtheorem{lemma}[theorem]{Lemma}

\newcommand{\N}{\mathbb{N}} 
\newcommand{\R}{\mathbb{R}} 
\renewcommand{\qed}{\hfill$\square$} 
\newcommand{\mC}{\mathcal{C}}
\newcommand{\mH}{\mathcal{H}}
\newcommand{\mW}{\mathcal{W}}
\newcommand{\mM}{\mathcal{M}}
\newcommand{\mE}{\mathcal{E}}
\newcommand{\mP}{\mathcal{P}}
\newcommand{\mT}{\mathcal{T}}
\newcommand{\eps}{\varepsilon}

\newcommand{\prob}{\mathbb{P}}

\tikzset{vtx/.style={inner sep=1.7pt, outer sep=0pt, circle, fill,draw}}

\usepackage{lineno}

\title{New bounds on the generalized Ramsey number $f(n,5,8)$} 

\author{Enrique Gomez-Leos\thanks{Department of Mathematics, Iowa State University, Ames IA.\newline \indent\,\,\,\,Email: \texttt{\{enriqueg, eheath, abparker, cschwi, zerbib\}@iastate.edu}} \and 
Emily Heath\footnotemark[1] \and 
Alex Parker\footnotemark[1] \and
Coy Schwieder\footnotemark[1] \and
Shira Zerbib\footnotemark[1]
}
\date{\today}

\begin{document}

\maketitle

\begin{abstract}
Let $f(n,p,q)$ denote the minimum number of colors needed to color the edges of $K_n$ so that every copy of $K_p$ receives at least $q$ distinct colors. In this note, we show  $\frac{6}{7}(n-1)\le f(n,5,8) \le n+o(n)$. The upper bound is proven using the ``conflict-free hypergraph matchings method" which was recently used by Mubayi and Joos to prove $f(n,4,5)=\frac{5}{6}n+o(n)$.  
\end{abstract}

\section{Introduction}

Given graphs $H$ and $G$ and positive integer $q$, an \emph{$(H,q)$-coloring} of $G$ is an edge-coloring in which each copy of $H$ receives at least $q$ colors. We denote by $f(G,H,q)$ the minimum number of colors required for an $(H,q)$-coloring of $G$. When $G=K_n$ and $H=K_p$, we use the notation $f(n,p,q)$. Note that determining $f(n,p,2)$ for all $n,p$ is equivalent to determining the classical multicolor Ramsey numbers. 
Introduced by Erd\H{o}s and Shelah~\cite{erdos1975,erdos1981}, these numbers were further explored by Erd\H{o}s and Gy\'arf\'as~\cite{EG} in the case where $G$ and $H$ are complete graphs and by Axenovich, F\"uredi, and Mubayi~\cite{AFM} in the case where $G$ and $H$ are complete bipartite graphs. Since then, the problem has been studied by many researchers, including~\cite{axenovich2000,BEHK,BCDP,56,CH1,CH2,CFLS,FPS,Mubayi1,mubayi2004,PS,sarkozy2000edge,sarkozy2003application}.

Among other results, Erd\H{o}s and Gy\'arf\'as used a probabilistic argument to give a general upper bound on $f(n,p,q)$, showing
\[f(n,p,q)=O\left(n^{\frac{p-2}{\binom{p}{2}-q+1}}\right).\]
Using a randomized coloring process and the differential equation method, Bennett, Dudek, and English~\cite{BDE} later improved this bound by a logarithmic factor for values of $q$ and $p$ with $q\leq (p^2-26p+55)/4$. Recently, Bennett, Delcourt, Li, and Postle~\cite{BDLP} extended this result to all values of $p$ and $q$ except at the values $q=\binom{p}{2}-p+2$ and $q=\binom{p}{2}-\left\lfloor\frac{p}{2}\right\rfloor+2$, where the local lemma bound of Erd\H{o}s and Gy\'arf\'as is known to be tight. They showed that for an $n$-vertex graph $G$, 
\[f(G,H,p)=O\left(\left(\frac{n^{|V(H)|-2}}{\log n}\right)^{\frac{1}{|E(H)|-q+1}}\right).\]
Moreover, they generalized this result to give an analogous upper bound for list colorings and for hypergraphs of higher uniformity. 

The proof in~\cite{BDLP} uses a new method introduced independently by Delcourt and Postle~\cite{DP} and by Glock, Joos, Kim, K\"uhn, and Lichev~\cite{GJKKL} for finding ``forbidden submatchings" or ``conflict-free hypergraph matchings."  
This method has been applied to a wide variety of problems; see for example~\cite{BHZ,DP, DP2, GJKKL, GJKKLP}.
In particular, Joos and Mubayi~\cite{JM} recently used this method to show that \[f(n,4,5)=\frac{5}{6}n+o(n).\] This result had previously been obtained by Bennett, Cushman, Dudek, and Pra\l at~\cite{BCDP} using a randomized coloring process involving the differential equation method, and  
answers a question of Erd\H{o}s and Gy\'arf\'as~\cite{EG}. 
They also used similar methods to prove new upper bounds showing that $f(K_{n,n},C_4,3)=\frac{2}{3}n+o(n)$ and $f(K_n,C_4,3)=\frac{1}{2}n+o(n)$. In this paper, we adapt their technique to improve the upper bound of $f(n,5,8)=O(n)$ given by Erd\H{o}s and Gy\'arf\'as~\cite{EG}.

\begin{theorem}\label{thm:upperbound}
    We have $f(n,5,8)\leq n+o(n)$.
\end{theorem}

To prove this theorem, we construct a coloring of $K_n$ in two stages. In the first stage, we define hypergraphs $\mH$ and $\mC$ for which a $\mC$-free matching in $\mH$ corresponds to a partial edge-coloring of $K_n$ with $n$ colors in which each color class consists of vertex-disjoint copies of edges and paths of length 2. In the second stage, we randomly color the remaining edges with a set of $o(n)$ colors and show that each copy of $K_5$ in the resulting coloring of $K_n$ receives at least 8 distinct colors.

The new ingredient in our proof, which makes it different than the proof in \cite{JM}, is that for every vertex $x\in V(K_n)$ and color $i$, we create two copies $x_i$ and $x_i'$ of $x$, and only one of them will be included in the vertices of our hypergraph $\mH$, where $x_i$ is included with probability $p$ and  $x_i'$ with probability $1-p$ for some fixed constant $p$. If $x_i'$ is included in $V(\mH)$, then $x$ will not be  allowed to have an incident edge colored $i$, and therefore $x$ will be an isolated vertex in the color class $i$.   

In addition, we obtain an improved lower bound on $f(n,5,8)$. Previously the best known bound was $f(n,5,8)\geq \frac{1}{3}(n-1)$ given by Erd\H{o}s and Gy\'arf\'as~\cite{EG}.

\begin{theorem}\label{thm:lowerbound}
    We have $f(n,5,8)\geq\frac{6}{7}(n-1)$.
\end{theorem}

In Section~\ref{sec:LB}, we prove this lower bound.  
In Section~\ref{sec:blackbox}, we introduce the main tool for our upper bound, namely, the conflict-free hypergraph matching method. Finally, in Section~\ref{sec:UB}, we use this technique to improve the upper bound on $f(n,5,8)$.

\section{Lower bound}\label{sec:LB}

In this section we prove Theorem \ref{thm:lowerbound}. Consider an arbitrary edge-coloring of $K_n$ using $m$ colors in which every copy of $K_5$ receives at least 8 colors. Note that each maximal component in any color class can have at most 3 edges, so it must be one of  $K_2, P_3, P_4, K_{1,3},$ or $K_3$. Define
    \begin{align*}
        A &= \text{the set of maximal monochromatic components isomorphic to } K_2, \\
        B &= \text{the set of maximal monochromatic components isomorphic to }  P_3, \\
        C &= \text{the set of maximal monochromatic components isomorphic to } P_4, \\
        D &= \text{the set of maximal monochromatic components isomorphic to } K_{1,3}, \\
        E &= \text{the set of maximal monochromatic components isomorphic to } K_3. 
    \end{align*}

Next, we partition $B$ into two sets as follows:
    \begin{align*}
        B_1 &= \{X \in B : |V(X) \cap V(Y)| \leq 1 \text{ for all } Y \in B - \{X\}\}\\
        B_2 &= \{X \in B : |V(X) \cap V(Y)| \geq 2 \text{ for some } Y \in B - \{X\}\}
    \end{align*}
By  definition, we have $B_2 = B - B_1$. \\

Observe that if there is a copy $X$ of $K_3$ in color $i$, then there can be no other edges of color $i$, and moreover, the edges between $V(X)$ and $V(K_n)-V(X)$ are all of distinct colors. 
Therefore, if there is a monochromatic copy of $K_3$, then there are at least $3(n-3)+1$ colors. 

So, we may assume there is no monochromatic triangle in this coloring. Thus, we have 
\begin{equation}\label{eq:LB1}
\binom{n}{2} = |A| + 2|B| + 3|C| + 3|D|.
\end{equation}

\begin{lemma} We have
    \begin{equation}\label{eq:LB2}
    |A| \geq |B_1| + \frac{1}{2} |B_2| + 2|C| + 3|D|.
    \end{equation}
\end{lemma}
\begin{proof}
For a graph $H$, let $L(H)$ be the set of vertex pairs of $H$ not forming an edge in $H$. 
Note that for any $X \in C \cup D$, $L(X)$ must contain three elements from $A$ (of three different colors), since otherwise we get a copy of $K_5$ with less than 8 colors.
For the same reason, for any $X \in D$ and $Y \in B \cup C \cup D$, then $X, Y$ share at most one vertex, and therefore $L(X)\cap L(Y)=\emptyset$.

Suppose $X = x_1x_2x_3x_4$ is an element of $C$.
We claim that for any $Y \in B \cup C \cup D$, $(L(X) - x_1x_4) \cap L(Y)=\emptyset$. Indeed, if not, then one of the paths $x_1x_2x_3$ or $x_2x_3x_4$ shares two vertices with $Y$, which results in a $K_5$ seeing less than 8 colors. 

Observe that for any $X \in B_1$, $L(X)$ must contain an element from $A$, since otherwise we get a $K_5$ with less than 8 colors. Note also that for any $X, Y \in B_2$, $L(X \cup Y)$ must be two elements from $A$. 

Now let $X, Y \in B_2$ share two vertices. (Hence $X\cup Y$ contains four vertices.) Then there is no $Z \in B_2$ such that $Z$ shares two vertices with either $X$ or $Y$, since otherwise we get a $K_5$ with less than 8 colors. That is, for any $X \in B_2$, there is a unique $Y \in B_2$ such that $X$ and $Y$ share two vertices. Therefore, the components in $B_2$ come in pairs $(X_1, Y_1), \dots, (X_{|B_2|/2}, Y_{|B_2|/2})$, where $|V(X_i) \cap V(Y_i)| = 2$ for all $1 \leq i \leq |B_2|/2$. Further, for $i \neq j$, $|V(Z_i) \cap V(Z_j)| \leq 1$, where $Z_i \in \{X_i, Y_i\}, Z_j \in \{X_j, Y_j\}$. 

Next, observe that for any pair $(X_i, Y_i)$, one of the maximal components isomorphic to $K_2$ must be in $L(X_i)$ or $L(Y_i)$. Denote this edge by $e_i$. We claim that $e_i \not \in L(X_j \cup Y_j)$ for all $j \neq i$. Suppose there exists some $1 \leq i, j \leq |B_2|/2$ such that $e_i \in L(X_j \cup Y_j)$. Then, since $|V(X_j \cup Y_j)| = 4$ and $e_i \in L(X_j \cup Y_j)$, $X_i$ shares two vertices with $X_j \cup Y_j$. Therefore, $|V(X_i \cup X_j \cup Y_j)| = 5$ and there are three maximal monochromatic components isomorphic to $P_3$ in the induced graph on these five vertices, creating a copy of $K_5$ with less than 8 colors, a contradiction.

We conclude the following: 
\begin{itemize}
    \item we can associate 3 edges in $A$ to every $X \in D$, namely the edges in $L(X)$, 
    \item we can associate 2 edges in $A$ to every $X=x_1x_2x_3x_4 \in C$, namely the edges in $L(X)-\{x_1x_4\}$, 
    \item we can associate 1 edge in $A$ to every $X \in B_1$, namely the edge in $L(X)$,
    \item we can associate 1 edge to every pair $(X_i, Y_i) \in B_2$, namely the edge $e_i$.
\end{itemize}

Thus, we get
$$|A| \geq |B_1| + \frac{1}{2}|B_2| + 2|C| + 3|D|, $$ proving the lemma.
\end{proof}

\begin{lemma} We have
    \begin{equation}\label{eq:LB3}
    mn - 2 \binom{n}{2} \geq |D| - |B_1|.
    \end{equation}
\end{lemma}
\begin{proof}
We count the number of pairs $(v,i)$ such that a vertex $v$ is incident to an edge of color $i$, in two different ways. First, the number of such pairs is 
\begin{equation*}2|A| + 3|B| + 4|C| + 4|D| = 2|A| + 3(|B_1| + |B_2|) + 4|C| + 4|D|.\end{equation*} 
For a fixed vertex $v \in V(K_n)$, define
\begin{eqnarray*}
    D_v^3 = \{X \in D : d_X(v) = 3\} & \text{and} & C_v^2 = \{X \in C : d_X(v) = 2\}.
\end{eqnarray*}
Observe that $\sum_{v \in V(K_n)} |D_v^3| = |D|$ and $\sum_{v \in V(K_n)} |C_v^2| = 2|C|$.

Note that for any $X \in D_v^3$, the set of colors $S(X)$ appearing on the edges of $L(X)$ cannot appear on any edge incident to $v$. Also, for any $Y = y_1vy_2y_3 \in C_v^2$,  the color in $S(Y)$ appearing on the edge $y_1y_2$ does not appear on any edge incident to $v$. Moreover, for any pair $X\in D_v^3$ and  $Y \in  C_v^2$, we have $S(X) \cap S(Y) = \emptyset$ since otherwise, this would lead to a $K_5$ seeing fewer than 8 colors. Finally, for any pair $(X_i = x_1x_2x_3, Y_i = y_1y_2y_3)$ in $B_2$, the color appearing on $L(X)$  
does not appear on any edge incident to $x_2$ and the color appearing on $L(Y)$ does not appear on any edge incident to $y_2$. Furthermore, the color appearing on $L(X)$ is not a color in $S(X')$ for $X' \in D_{x_2}^3$ since this would lead to a $K_5$ seeing fewer than 8 colors. Similarly, for $X' = z_1x_2z_3z_4 \in C_{x_2}^2$, we get the same contradiction if the color appearing on $L(X)$ is the color which appears on $z_1z_3$. The same argument holds for the color on $L(Y)$. Therefore, for a fixed vertex $v$, each color not appearing at $v$ is counted exactly once in our above arguments. Otherwise, this would result in a $K_5$ seeing fewer than 8 colors. All together, this implies that the number of pairs $(v,i)$ such that a vertex $v$ is incident to an edge of color $i$ is at most  $\sum_{v \in V(K_n)}( m - 3|D_v^3| - |C_v^2|)  - |B_2|$.

Thus we obtain, 
\begin{equation*}
    \begin{split}
        2|A| + 3(|B_1| + |B_2|) + 4|C| + 4|D|
            &\leq \sum_{v \in V(K_n)} (m - 3|D_v^3| - |C_v^2|) - |B_2| \\
            &=  mn - 3|D| - 2|C| - |B_2|.
    \end{split}
\end{equation*}

By rearranging and applying (\ref{eq:LB1}), we get $mn - 2 \binom{n}{2} \geq |D| - |B_1|$ as desired.
\end{proof}

\begin{lemma} We have
    \begin{equation}\label{eq:LB4}
    mn - 2 \binom{n}{2} \geq -2|A| + 2|B_1| - 4|B_2| - 6|C| - 2|D|.
    \end{equation}
\end{lemma}
\begin{proof}
    By averaging, there exists a vertex $v$ that is adjacent to at least $3|B_1|/n$ elements from $B_1$. For each $X \in B_1$ with $X$ incident to $v$, we get two colors (one from $X$ and one from $L(X)$). Neither of these colors can be the same as two colors coming from some other $Y \in B_1$ with $Y$ incident to $v$. Also, we get four colors for each $X \in D_v^3$ (one from $X$ and three from the edges of $L(X)$). Further, none of these colors can be the same as any color coming from some other $Y \in D_v^3$ nor can they be the same as any color coming from $Y \in B_1$ with $Y$ incident to $v$. Therefore,  $m \geq 4|D_v^3| + 2\frac{3|B_1|}{n}$. Summing over all vertices, we get
    \begin{equation*} mn \geq 4|D| + 6|B_1|. \end{equation*}

    Combining with (\ref{eq:LB1}), this shows $mn - 2\binom{n}{2} \geq -2|A| + 2|B_1| - 4|B_2| - 6|C| - 2|D|$.
\end{proof}

 Now, setting $P = mn - 2\binom{n}{2}$ and combining equations~(\ref{eq:LB1}), (\ref{eq:LB2}), (\ref{eq:LB3}),  (\ref{eq:LB4}), we get the following linear program:
\begin{align*}
\text{min }   P                                                                        \\
\text{s.t.  ~} & \binom{n}{2}            =      |A| + 2(|B_1| + |B_2|) + 3(|C| + |D|)   \\
                   &  |A|                     \geq   |B_1| + \frac{1}{2}|B_2| + 2|C| + 3|D|  \\
                    & P  \geq   |D| - |B_1|                             \\
                    & P  \geq   -2|A| + 2|B_1| - 4|B_2| - 6|C| - 2|D|   \\
                   &  A, B_1, B_2, C, D, P   \in     \mathbb{Z}                              \\
                   &  A, B_1, B_2, C, D      \geq    0
\end{align*}
Solving this linear program, we obtain $mn - 2\binom{n}{2} =P \geq -\frac{2}{7}\binom{n}{2}$, and  thus we have  $m \geq \frac{6}{7}(n - 1),$ as needed.  

\qed

\section{Conflict-free hypergraph matching method}\label{sec:blackbox}

In order to prove our upper bound on $f(n,5,8)$, we will use the version stated in~\cite{JM} of the conflict-free hypergraph matching theorem from~\cite{GJKKL}. 

Given a hypergraph $\mH$ and a vertex $v\in V(\mH)$, its \emph{degree} $\deg_{\mH}(v)$ is the number of edges in $\mH$ containing $v$. The maximum degree and minimum degree of $\mH$ are denoted by $\Delta(\mH)$ and $\delta(\mH)$, respectively. For $j\geq 2$, $\Delta_j(\mH)$ denotes the maximum number of edges in $\mH$ which contain a particular set of $j$ vertices, over all such sets. 

In addition, for a (not necessarily uniform) hypergraph $\mC$ and an integer $k$, let $\mC^{(k)}$ be the set of edges in $\mC$ of size $k$. For a vertex $u\in V(\mC)$, let $\mC_u$ denote the hypergraph $\{C\backslash \{u\} \mid C\in E(\mC), u\in C\}$.

Given a hypergraph $\mH$, a hypergraph $\mC$ is a \emph{conflict system} for $\mH$ if $V(\mC)=E(\mH)$. A set of edges $E \subset \mH$ is \emph{$\mC$-free} if $E$ contains no subset $C\in \mC$. Given integers $d\geq 1$, $\ell\geq 3$, and $\eps\in(0,1)$, we say $\mC$ is \emph{$(d,\ell,\eps)$-bounded} if 
 $\mC$ satisfies the following conditions:
\begin{enumerate}
    \item[(C1)] $3\leq |C|\leq \ell$ for all $C\in\mC$;
    \item[(C2)] $\Delta(\mC^{(j)})\leq \ell d^{j-1}$ for all $3\leq j\leq \ell$;
    \item[(C3)] $\Delta_{j'}(\mC^{(j)})\leq d^{j-j'-\eps}$ for all $3\leq j\leq \ell$ and $2\leq j'\leq j-1$. 
\end{enumerate}

Finally, given a $(d,\ell,\eps)$-bounded conflict system $\mC$ for a hypergraph $\mH$, we will define a type of weight function which can be used to guarantee that the almost-perfect matching given by Theorem~\ref{thm:blackbox} below satisfies certain quasirandom properties. We say a function $w:\binom{\mH}{j}\rightarrow[0,\ell]$ for $j\in\N$ is a \emph{test function} for $\mH$ if $w(E)=0$ whenever $E\in\binom{\mH}{j}$ is not a matching, and we say $w$ is \emph{$j$-uniform}. For a function $w:A\rightarrow \R$ and a finite set $X\subset A$, let $w(X):=\sum_{x\in X} w(x)$. If $w$ is a $j$-uniform test function, then for each $E\subset \mH$, let $w(E)=w(\binom{E}{j})$. Given $j,d\in\N$, $\eps>0$, and a conflict system $\mC$ for hypergraph $\mH$, we say a $j$-uniform test function $w$ for $\mH$ is \emph{$(d,\eps,\mC)$-trackable} if $w$ satisfies the following conditions:
\begin{enumerate}
    \item[(W1)] $w(\mH)\geq d^{j+\eps}$;
    \item[(W2)] $w(\{{E\in \binom{\mH}{j}:E\supseteq E'\})\leq w(\mH})/d^{j'+\eps}$ for all $j'\in[j-1]$ and $E'\in\binom{\mH}{j'}$;
    \item[(W3)] $|(\mC_e)^{(j')}\cap(\mC_f)^{(j')}|\leq d^{j'-\eps}$ for all $e,f\in \mH$ with $w(\{E\in\binom{\mH}{j}:e,f\in E\})>0$ and all $j'\in[\ell-1]$;
    \item[(W4)] $w(E)=0$ for all $E\in\binom{\mH}{j}$ that are not $\mC$-free.
\end{enumerate}

\begin{theorem}[\cite{GJKKL}, Theorem 3.3]\label{thm:blackbox}
For all $k,\ell\geq 2$, there exists $\eps_0>0$ such that for all $\eps\in(0,\eps_0)$, there exists $d_0$ such that the following holds for all $d\geq d_0$. Suppose $\mH$ is a $k$-regular hypergraph on $n\leq \exp(d^{\eps^3})$ vertices with $(1-d^{-\eps})d\leq \delta(\mH)\leq \Delta(\mH)\leq d$ and $\Delta_2(\mH)\leq d^{1-\eps}$. Suppose $\mC$ is a $(d,\ell,\eps)$-bounded conflict system for $\mH$, and suppose $\mW$ is a set of $(d,\eps, \mC)$-trackable test functions for $\mH$ of uniformity at most $\ell$ with $|\mW|\leq \exp(d^{\eps^3})$. Then, there exists a $\mC$-free matching $\mM\subset \mH$ of size at least $(1-d^{-\eps^3})n/k$ with $w(\mM)=(1\pm d^{-\eps^3})d^{-j}w(\mH))$ for all $j$-uniform $w\in \mW$. 
\end{theorem}

We will say that a hypergraph $\mH$ with $(1-d^{-\eps})d\leq \delta(\mH)\leq \Delta(\mH)\leq d$  is \emph{almost $d$-regular.}

In addition, we will use the Lov\'asz Local Lemma~\cite{AS}. For a set of events $\mE$ and a graph $G$ on vertex set $\mE$, we say that $G$ is a \emph{dependency graph} for $\mE$ if each event $E\in\mE$ is mutually independent from the family of events which are not adjacent to $E$ in $G$.

\begin{lemma}[Lov\'asz Local Lemma]\label{lem:LLL}
Let $\mE$ be a finite set of events in a probability space $\Theta$ and let $G$ be a dependency graph for $\mE$. Let $N(E)$ denote the neighborhood of $E$ in $G$ for each $E\in \mE$. Suppose there is an assignment $x:\mE\rightarrow[0,1)$ of real numbers to $\mE$ such that for all $E\in\mE$, we have 
\begin{equation}\label{LLL}
    \prob(E)\leq x(E)\prod_{E'\in N(E)} (1-x(E')).
\end{equation}
Then, the probability that none of the events in $\mE$ happens is 
\[\prob\left(\bigcap_{E\in\mE}\bar{E}\right)\geq\prod_{E\in\mE}(1-x(E))>0.\]
\end{lemma}

We will also need the following concentration inequality. 

\begin{theorem}(McDiarmid's inequality~\cite{mcdiarmid})\label{thm:mcdiarmid}
Suppose $X_1, \dots, X_m$ are independent random variables. Suppose $X$ is a a real-valued random variable determined by $X_1, \dots, X_m$ such that changing the outcome of $X_i$ changes $X$ by at most $b_i$ for all $i \in [m]$. Then, for all $t >0$, we have 
\begin{align*}
    \mathbb{P}[|X-\mathbb{E}[X] | \geq t] \leq 2 \operatorname{exp}\Bigl( -\cfrac{2t^2}{\sum_{i \in [m]}b_i^2}\Bigr).
\end{align*}
\end{theorem}

\section{Upper bound}\label{sec:UB}

Our coloring process occurs in two stages. 
The first coloring uses $n$ colors to color a majority of the edges of $K_n$. This coloring is defined by constructing appropriate hypergraphs $\mH$ and $\mC$ for which a $\mC$-free matching in $\mH$ corresponds to a partial coloring of $K_n$ which contains no ``conflicts" $C\in \mC$. In particular, this coloring will be a tiling of a subgraph of $K_n$ with 2-colored triangles, where no two  triangles which intersect in a vertex share a color. 

We will say that a copy of $K_5$ is \emph{bad} if it is colored with at most 7 colors. In the following lemma, we show that each bad $K_5$ in our partial coloring must contain one of a handful of \emph{bad subgraphs}, which we refer to as having \emph{type} $t\in\{a,b,c,d,e,f\}$.

\begin{lemma}\label{lem:badK5}
Let $f:E(K_n)\rightarrow C$ be an edge-coloring where every color class consists of vertex-disjoint edges and 2-edge paths, and any two monochromatic 2-edge paths share at most one vertex. In addition, assume any two 2-colored triangles which share a vertex must have disjoint sets of colors. 
Then every bad $K_5$ contains one of the following types of bad subgraphs (shown in Figure~\ref{fig:forbidden}):
\begin{enumerate}[(a)]
    \item An alternating $C_4$ formed by two monochromatic matchings,
    \item An alternating $C_5$ with one monochromatic matching and a second color on a disjoint edge and path,
    \item The subgraph $Q$ which contains a two-colored triangle and consists of two monochromatic matchings and one monochromatic path, 
    \item A subgraph consisting of one monochromatic matching and two monochromatic paths, 
    \item A subgraph consisting of two monochromatic matchings and one monochromatic path, or
    \item A subgraph consisting of three monochromatic matchings.
\end{enumerate}
\end{lemma}

\begin{figure}
    \centering    
    \begin{subfigure}[b]{0.3\textwidth}
        \centering
        \begin{tikzpicture}[scale = 0.75]
        
            \draw \foreach \x  in {0,1,...,4}{(90*\x+45:1.5)node[vtx](\x){}};
            
            \draw[ultra thick, blue] (0)--(1)  node[midway, above] {$\ell$};
            \draw[ultra thick] (1)--(2) node[midway, left] {$i$};
            \draw[ultra thick, blue] (2)--(3) node[midway, below] {$\ell$};
            \draw[ultra thick] (3)--(0) node[midway, right] {$i$};
            \end{tikzpicture}

        \caption{Alternating $C_4$}\label{C4}
    \end{subfigure}    
    \begin{subfigure}[b]{0.3\textwidth}
        \centering
        \begin{tikzpicture}[scale = 0.75]
        
            \draw \foreach \x  in {0,1,...,5}{(72*\x+90:1.5)node[vtx](\x){}};
            
            \draw[ultra thick] (0)--(1)  node[midway, above] {$i$};
            \draw[blue, ultra thick] (4)--(0) node[midway, above] {$\ell$};
            \draw[ultra thick, blue] (2)--(3) node[midway, below] {$\ell$};
            \draw[ultra thick] (3)--(4) node[midway, right] {$i$};
            \draw[ultra thick] (1)--(2) node[midway,left] {$i$};
            \end{tikzpicture} 
         
        \caption{Alternating $C_5$}
        \label{C5}
    \end{subfigure}
    \begin{subfigure}[b]{0.3\textwidth}
    \centering         \begin{tikzpicture}[scale = 0.75]
        
            \draw \foreach \x  in {0,1,...,5}{(72*\x+90:1.5)node[vtx](\x){}};
            
            \draw[ultra thick] (0)--(1)  node[midway, above] {$i$};
            \draw[ultra thick] (1)--(2) node[midway, left] {$i$};
            \draw[ultra thick, blue] (0)--(2) node[midway, left] {$\ell$};
            \draw[blue, ultra thick] (4)--(3) node[midway, right] {$\ell$};
            \draw[ultra thick, red] (2)--(3) node[midway,below]  {$m$};
            \draw[ultra thick, red] (4)--(1) node[midway,above]  {$m$};  
            \end{tikzpicture}
        \caption{Subgraph $Q$}\label{pentagon}
    \end{subfigure}

    \vspace{.5cm}

     \begin{subfigure}[b]{\textwidth}
     \centering
        \begin{tikzpicture}[scale = 0.75]

            \draw \foreach \x  in {0,1,...,5}{(72*\x+90:1.5)node[vtx](\x){}};
            
            \draw[ultra thick] (0)--(1)  node[midway, above] {$i$};
            \draw[ultra thick] (1)--(2) node[midway, left] {$i$};
            \draw[ultra thick, blue] (1)--(4) node[midway, below] {$\ell$};
            \draw[blue, ultra thick] (3)--(4) node[midway, right] {$\ell$};
            \draw[ultra thick, red] (2)--(3) node[midway,below]  {$m$};
            \draw[ultra thick, red] (4)--(0) node[midway,above]  {$m$};
        \end{tikzpicture}
    \hspace{.2in} 
    \begin{tikzpicture}[scale = 0.75]
        
            \draw \foreach \x  in {0,1,...,5}{(72*\x+90:1.5)node[vtx](\x){}};
            
            \draw[ultra thick] (0)--(1)  node[midway, above] {$i$};
            \draw[ultra thick] (1)--(2) node[midway, left] {$i$};
            \draw[ultra thick, blue] (0)--(4) node[midway, above] {$\ell$};
            \draw[blue, ultra thick] (3)--(4) node[midway, right] {$\ell$};
            \draw[ultra thick, red] (1)--(4) node[midway,below]  {$m$};
            \draw[ultra thick, red] (2)--(3) node[midway,below]  {$m$};   
        \end{tikzpicture}
        \hspace{.2in} 
    \begin{tikzpicture}[scale = 0.75]
        
            \draw \foreach \x  in {0,1,...,5}{(72*\x+90:1.5)node[vtx](\x){}};
            
            \draw[ultra thick] (0)--(1)  node[midway, above] {$i$};
            \draw[ultra thick] (1)--(2) node[midway, left] {$i$};
            \draw[ultra thick, blue] (1)--(4) node[midway, above] {$\ell$};
            \draw[blue, ultra thick] (1)--(3) node[midway, below] {$\ell$};
            \draw[ultra thick, red] (2)--(3) node[midway,below]  {$m$};
            \draw[ultra thick, red] (4)--(0) node[midway,above]  {$m$};      
            \end{tikzpicture}
        \hspace{.2in}
        \begin{tikzpicture}[scale = 0.75]
            \draw \foreach \x  in {0,1,...,5}{(72*\x+90:1.5)node[vtx](\x){}};
            
            \draw[ultra thick] (0)--(1)  node[midway, above] {$i$};
            \draw[ultra thick] (1)--(2) node[midway, left] {$i$};
            \draw[ultra thick, blue] (0)--(4) node[midway, above] {$\ell$};
            \draw[blue, ultra thick] (4)--(3) node[midway, right] {$\ell$};
            \draw[ultra thick, red] (1)--(3) node[midway,above]  {$m$};
            \draw[ultra thick, red] (2)--(4) node[midway,above]  {$m$};   
            \draw[ultra thick, white] (2)--(3) node[midway,below]  {$m$};
            \end{tikzpicture}
    \caption{One monochromatic matching and two monochromatic paths}\label{2M1P}
    \end{subfigure}
    
    \vspace{0.5cm}

    \begin{subfigure}[b]{\textwidth}
        \centering
        \begin{tikzpicture}[scale = 0.75]
        
            \draw \foreach \x  in {0,1,...,5}{(72*\x+90:1.5)node[vtx](\x){}};
            
            \draw[ultra thick] (0)--(1)  node[midway, above] {$i$};
            \draw[ultra thick] (1)--(2) node[midway, left] {$i$};
            \draw[ultra thick, blue] (1)--(4) node[midway, below] {$\ell$};
            \draw[blue, ultra thick] (0)--(3) node[midway, right] {$\ell$};
            \draw[ultra thick, red] (2)--(3) node[midway,below]  {$m$};
            \draw[ultra thick, red] (4)--(0) node[midway,above]  {$m$};  
            \end{tikzpicture}
        \hspace{.2in}
        \begin{tikzpicture}[scale = 0.75]
        
            \draw \foreach \x  in {0,1,...,5}{(72*\x+90:1.5)node[vtx](\x){}};
            
            \draw[ultra thick] (0)--(1)  node[midway, above] {$i$};
            \draw[ultra thick] (1)--(2) node[midway, left] {$i$};
            \draw[ultra thick, blue] (0)--(4) node[midway, above] {$\ell$};
            \draw[blue, ultra thick] (1)--(3) node[midway, above] {$\ell$};
            \draw[ultra thick, red] (2)--(3) node[midway,below]  {$m$};
            \draw[ultra thick, red] (4)--(1) node[midway,above]  {$m$};  
            \end{tikzpicture}
        \caption{Two monochromatic matchings and one monochromatic path}\label{1M2P}
    \end{subfigure} \hfil
    
    \begin{subfigure}[b]{\textwidth}
        \centering
        \begin{tikzpicture}[scale = 0.75]
        
            \draw \foreach \x  in {0,1,...,5}{(72*\x+90:1.5)node[vtx](\x){}};
            
            \draw[ultra thick] (0)--(1)  node[midway, above] {$i$};
            \draw[ultra thick] (2)--(3) node[midway, below] {$i$};
            \draw[ultra thick, blue] (1)--(2) node[midway, left] {$\ell$};
            \draw[ultra thick, blue] (3)--(4) node[midway, right] {$\ell$};
            \draw[ultra thick, red] (1)--(3) node[midway,above]  {$m$};
            \draw[ultra thick, red] (4)--(0) node[midway,above]  {$m$};  
            \end{tikzpicture}
        \hspace{.2in}
                \begin{tikzpicture}[scale = 0.75]
        
            \draw \foreach \x  in {0,1,...,5}{(72*\x+90:1.5)node[vtx](\x){}};
            
            \draw[ultra thick] (0)--(1)  node[midway, above] {$i$};
            \draw[ultra thick] (2)--(3) node[midway, below] {$i$};
            \draw[ultra thick, blue] (1)--(2) node[midway, left] {$\ell$};
            \draw[ultra thick, blue] (3)--(4) node[midway, right] {$\ell$};
            \draw[ultra thick, red] (1)--(4) node[midway,below]  {$m$};
            \draw[ultra thick, red] (3)--(0) node[midway,right]  {$m$};  
            \end{tikzpicture}
        \caption{Three monochromatic matchings}\label{3M}
    \end{subfigure}
    
    \caption{Bad subgraphs in $K_n$ which correspond to edges in the conflict hypergraph $\mC$}
    \label{fig:forbidden}
\end{figure}

\begin{proof}
Fix a bad copy $K$ of $K_5$, and suppose it does not contain a type $a$ bad subgraph. (That is, $K$ has no 2-colored copy of $C_4$.) We will show that $K$ must contain one of the other five types of bad subgraphs.

In order for $K$ to receive at most 7 colors on its 10 edges, there must be at least three color repeats. Since every color class contains vertex-disjoint edges and 2-edge paths, each color can appear at most 3 times on $K$. Moreover, if a color $i\in C$ appear three times, then it must appear on a 2-edge path and an edge. In this case, any other color repeat must appear on a monochromatic 2-edge matching, because any two monochromatic 2-edge paths can only intersect in one vertex and any two triangles which share a vertex have disjoint colors. Therefore, $K$ contains a subgraph of type $b$.

Thus, we may assume that each color appears at most twice on $K$. So, $K$ must contain at least three monochromatic 2-edge matchings or paths. 
Note that since any two monochromatic 2-edge paths can intersect in at most one vertex, any three such paths would require at least six vertices and hence cannot be contained in $K$. Therefore, $K$ must contain at least one monochromatic 2-edge matching and at most two monochromatic 2-edge paths, resulting in one of the bad subgraphs of type $c$, $d$, $e$, or $f$. 

Moreover, it can be easily shown in these cases that $K$ must contain one of the nine subgraphs in Figure~\ref{fig:forbidden}. For example, any bad subgraph of type $f$ must contain three monochromatic 2-edge matchings, say in colors $i,\ell,m$. Since there are no monochromatic 2-edge paths and no alternating $C_4$ in this case, the subgraph with edges of colors $i$ and $\ell$ must form an alternating path $P_5$. Then there are only two non-isomorphic ways to place the matching of color $m$ without creating an alternating $C_4$. The other types of bad subgraphs can be checked similarly.
\end{proof}

We will define a hypergraph $\mH$ and conflict system $\mC$ so that the edges of $\mC$ correspond to bad subgraphs in $K_n$. Our key technical result below, Theorem~\ref{thm:coloringproperties}, guarantees that our choices of $\mH$ and $\mC$ satisfy the requirements of Theorem~\ref{thm:blackbox} needed to give a $\mC$-free matching of $\mH$ which corresponds to a partial $(5,8)$-coloring of $K_n$. 

To color the remaining edges of $K_n$, we will apply a random coloring using a set of $n^{1-\delta}$ new colors. Properties (\ref{property4}) and (\ref{property5}) of Theorem~\ref{thm:coloringproperties} allow us to use the Lov\'asz Local Lemma to show that the resulting union of these two colorings is a $(5,8)$-coloring of $K_n$. 

In order to state Theorem~\ref{thm:coloringproperties}, we need some additional terminology. Given a partial edge-coloring of $K_n$, we say a set of uncolored edges $E'\subset E(K_n)$ \emph{completes} a bad subgraph of type $t\in \{a,b,c,d,e,f\}$ if there is a way to assign colors to $E'$ which would create a bad subgraph of type $t$. In particular, we will be interested in the cases where an edge or a 2-edge matching completes a bad subgraph.

\begin{theorem}\label{thm:coloringproperties}
There exists $\delta>0$ such that for all sufficiently large $n$ in terms of $\delta$, there exists an edge-coloring of a subgraph $F\subset K_n$ with at most $n$ colors and the following properties:
\begin{enumerate}[(I)]
    \item Every color class consists of vertex-disjoint edges and 2-edge paths. \label{property1}
    \item For all triangles $xyz$ in $K_n$ where $xy$ and $yz$ receive the same color $i$ and $xz$ is colored $\ell$, the vertex $y$ is an isolated vertex in color class $\ell$, and $xz$ forms a component in color class $\ell$. \label{property2}
    \item There are no bad subgraphs in $F$. \label{property3}
    \item The graph $L=K_n-E(F)$ has maximum degree at most $n^{1-\delta}$. \label{property4}
    \item For each uncolored edge $xy\in E(L)$ and for each type $t\in\{a,b,c,d,e,f\}$ of bad subgraph, there are at most $n^{1-\delta}$ edges $x'y'\in E(L)$ with $\{x,y\}\cap \{x',y'\}=\emptyset$ for which $\{xy,x'y'\}$ completes a bad subgraph of type $t$. \label{property5}
\end{enumerate}
\end{theorem}

\subsection{Proof of Theorem~\ref{thm:coloringproperties}}

 For clarity, we use $k$ throughout the proof when discussing the number of colors to distinguish between counting vertices and counting colors. 

We start by constructing a random vertex set as follows. Let $W = \bigcup_{i \in [k]} W_i$, where $W_1, \dots, W_k$ are disjoint copies of $V(K_n)$. Initially, for $i \in [k]$, set $V_i = V_i' = \emptyset$. Now, for each vertex $w_i$ in $W_i$, independently with probability $p=1/6$ add a copy $v_i'$ of $w_i$ to the set $V_i'$. Otherwise, with probability $1-p$, add a copy $v_i$ of $w_i$ to the set $V_i$. Let $V=\bigcup_{i\in[k]}V_i$ and $V'=\bigcup_{i\in[k]}V_i'$. Note that while this process is similar to the one used by Joos and Mubayi~\cite{JM}, it differs in that we add vertices to a new set $V'$ with probability $p$ rather than simply deleting the vertices. 

Now, we construct our 9-uniform hypergraph $\mH$ with vertex set $E(K_n) \cup V \cup V'$. For each triangle $K=uvw$ in $K_n$ and each pair of distinct colors $i, \ell \in [k]$, we add the edge
\[\{uv, uw, vw, u_i, v_i, w_i, v_{\ell}, w_{\ell}, u_{\ell}'\}\]
to $\mH$ if $u_i, v_i, w_i, v_{\ell}, w_{\ell} \in V$ and $u_{\ell}' \in V'$. We will denote this edge in $\mH$ by $e=(K,i,\ell)$. 
Note that a matching in $\mH$ corresponds to a collection of edge-disjoint triangles $uvw$ in $K_n$, where $uv$ and $uw$ have color $i$ and $vw$ has color $\ell$, and where no other 2-colored triangle containing $u$ uses color $\ell$ (because only one of the two vertices $u_{\ell}$ and $u_{\ell}'$ exists in $V(\mH)$). Thus, a matching in $\mH$ yields a partial coloring of $K_n$ which satisfies properties~(\ref{property1}) and (\ref{property2}) of Theorem~\ref{thm:coloringproperties}. 

In order to achieve properties~(\ref{property3})-(\ref{property5}), we will later define an appropriate conflict hypergraph $\mC$ and trackable test functions for $\mH$. 

But first, we  check that the degree and codegree conditions for $\mH$ needed to apply Theorem~\ref{thm:blackbox} are satisfied. In particular, we will show that $\mH$ is essentially $d$-regular, where $d=\frac{5^5}{2\cdot 6^5}n^3$. Then, it is easy to check that $\Delta_2(\mH)=O(d/n)< d^{1-\eps}$ for $\eps\in(0,\frac{1}{4})$. 
As in~\cite{JM}, each vertex in $\mH$ of the form $uv\in E(K_n)$ has expected degree 
\[\mathbb{E}[d_{\mH}(uv)]=(n-2)\cdot k(k-1)\cdot 3\cdot (1-p)^5p\]
and each  vertex of the form $u_i\in V_i$ has expected degree 
\[\mathbb{E}[d_{\mH}(u_i)]=\left(\binom{n-1}{2}+(n-1)(n-2)+(n-1)(n-2)\right)\cdot (k-1)\cdot (1-p)^4p.\] 
In addition, we have that the expected degree of a vertex of the form $u_i'$ is \[\mathbb{E}[d_{\mH}(u_i')]=\binom{n-1}{2}\cdot (k-1)\cdot (1-p)^5.\] 
By our choices of $p=1/6$ and $k=n$, each vertex in $\mH$ has degree $d\approx \frac{5^5}{2\cdot 6^5}n^3.$ 

We can apply McDiarmid's inequality to show concentration of these degrees for each type of vertex in $\mH$. First,  fix $uv\in E(G)$ and consider how the degree of $uv$ in $\mH$ will change if some $w\in V$ is instead placed in $V'$ (or vice versa). For each $w\in V\cup V'$, let $b_w=n^2$ if $w$ is a copy of $u$ or $v$ and $b_w=n$ otherwise. Thus, $\sum_w b_w^2=O(n^5)$. Now instead fix $u_i\in V_i$ or $u_i'\in V_i'$. For each $w\in V\cup V'\backslash \{u_i\}$, let $b_w=n^2$ if $w$ is a copy of $u$ or $w\in V_i\cup V_i'$, and let $b_w=n$ otherwise. Again, we have $\sum_w b_w^2=O(n^5)$, so in either case, applying Theorem~\ref{thm:mcdiarmid} shows concentration on an interval of length $O(n^{5/2})$. 
Thus, we have with high probability that
\begin{align*}
   \frac{5^5}{2\cdot 6^6}n^3 - O(n^{8/3})\leq \delta(\mH) \leq \Delta(\mH) \leq \frac{5^5}{2\cdot 6^6}n^3 + O(n^{8/3}).
\end{align*}

In addition, McDiarmid's inequality implies that for any two $u_i,v_i\in V$, there are $O(n^2)$ edges containing both $u_i$ and $v_i$ with high probability. To see this, define $b_w$ for $w\in V\backslash \{u_i,v_i\}$ by $b_w=n$ if $w$ is a copy of $u$ or $v$ or if $w\in V_i$, and by $b_w=1$ otherwise. 

We will also need to show that another quantity is close to its expected value, but for now, we assume that this is the case and fix a choice for $\mH$ with these properties to refer to as the deterministic 9-uniform hypergraph $\mH$. We set $d=\Delta(\mH)$, so $\mH$ is essentially $d$-regular and $d=\Theta(n^3)$. 

We now define a hypergraph $\mC$ with vertex set $E(\mH)$ and edges of size 4, 5, and 6 which is a conflict system for $\mH$. The edges of $\mC$ correspond to bad subgraphs in $K_n$ which arise from 4, 5, or 6 triangles in $K_n$ that form a matching in $\mH$.
More precisely, given a bad subgraph $H$ of type $t$, we include $E_H=\{e_{xy}:xy\in E(H)\}$ in $\mC$ if $E_H$ is a matching in $\mH$.  We call this edge of $\mC$ a \emph{conflict} of type $t$. 
It is easy to verify that for every type of bad subgraph $H$, $4\le |E_H|\le 6$. 
Indeed, this follows from the fact that any monochromatic path of length 2 in $H$ corresponds to a single element in $E_H$, every 2-colored triangle in $H$ corresponds to a single element in $E_H$, and every monochromatic matching of size 2 corresponds to a two elements in $E_H$.
For every graph edge $xy$ in $H$, let $e_{xy}$ be the $\mH$-edge in $E_H$ containing the element $xy$.


We now check the degree conditions needed to show that the conflict hypergraph $\mC$ is $(d,O(1),\eps)$-bounded for all $\eps\in(0,\frac{1}{4})$. Condition (C1) is met since $4\leq |C|\leq 6$ for all $C\in\mC$.  

For condition (C2), we consider the maximum degree in $\mC^{(j)}$ for each $4\leq j\leq 6$. To this end, fix $e=(K,i,\ell)\in V(\mC)$ with $K=xyz$ where $xy$ and $yz$ receive color $i$ and $xz$ receives color $\ell$. To count the conflicts of type $a$ or $b$ containing $e$, note that there are $O(n^4)$ ways to pick a second edge in $\mH$ containing color $i$ or $\ell$, say $e'=(K',i,m)$ with $K'=uvw$. Then, each of the other two $\mH$-edges in the conflict must contain a graph edge with one vertex in $\{x,y,z\}$ and the other in $\{u,v,w\}$, and these two $\mH$-edges must share at least one color, so there are $O(d)$ ways to pick a third edge in $\mH$ and $O(d/n)$ ways to pick a fourth edge in $\mH$ to complete the conflict of type $a$ or $b$. 

By similar reasoning, there are $O(n^9)$ conflicts of type $c$ in $\mC^{(4)}$ containing $e$ to which $e$ contributes a single graph edge. There are an additional $O(n^9)$ conflicts of type $c$ containing $e$ to which $e$ contributes all three graph edges, $xy$, $yz$, and $zx$. To see this, note that there are $O(n^4)$ ways to pick a second $\mH$-edge $K'$ containing the color $\ell$, after which the third and fourth $\mH$-edges in the conflict must each contain a graph edge with one vertex from $\{x,y,z\}$ and one vertex from $K'$. Thus, there are $O(d)$ choices for the third edge in the conflict, and since the third and fourth edges also share a color, there are $O(d/n)$ choices for the fourth edge to complete a type $c$ conflict. 

It remains to count the conflicts of type $t\in\{d,e,f\}$ which contain $e$. We may assume for any bad subgraph $H$ corresponding to a conflict $C$ of type $t$ which contains $e$ that $e$ contributes either the monochromatic path $xyz$ in color $i$ or the edge $xy$ in color $i$. Let $j$ be the size of a conflict of type $t$. To count the conflicts of type $t$ containing $e$ in the first way, note that there are there are $O(n^2)$ ways to pick the other two graph vertices in $H$, $\Delta(\mH)=O(d)$ ways each to pick another $j-3$ $\mH$-edges in $C$ which will be in matchings of distinct colors (since we know each of these $\mH$-edges must contain a particular graph edge), and $\Delta_2(\mH)=O(d/n)$ ways each to pick the last two $\mH$-edges to complete $C$ (since for each, we know either a graph edge and a color which it must contain or we know two graph edges which it must contain). Similarly, for edges in $\mC^{(j)}$ containing $e$ in the second way, there are $O(n^3)$ ways to pick the other graph vertices in $C$, $\Delta(\mH)=O(d)$ ways each to pick another $j-4$ $\mH$-edges in $C$ (since we know each of these $\mH$-edges must contain a particular graph edge), and $O(d/n)$ choices each for the last three $\mH$-edges (since we know for each of these either  a graph edge and color or two graph edges which it must contain).  Thus, $\Delta(\mC^{(j)})=O(d^{j-1})$ for each $4\leq j\leq 6$. 

Finally, condition (C3) can be verified using very similar arguments to bound the codegrees by $\Delta_{j'}(\mC^{(j)})=O(d^{j-j'}/n)<d^{j-j'-\frac{1}{4}}$ for all $4\leq j\leq 6$ and $2\leq j'\leq j-1$. Thus, $\mC$ is a $(d,O(1),\eps)$-bounded conflict system for $\mH$ for all $\eps\in(0,\frac{1}{4})$, as desired. 


In order to obtain property (\ref{property4}) of Theorem~\ref{thm:coloringproperties}, we define the following test functions and check that they are $(d, \eps, \mC)$-trackable. For each $v \in V(K_n)$, let $S_v \subset E(K_n)$ be the set of $n - 1$ edges in $K_n$ incident to $v$. Let $w_v: E(\mH) \to [0, 2]$ be the weight function that assigns every edge of $\mH$ the size of its intersection with $S_v$. Then, $w_v(\mH) = \sum_{e \in S_v} d_{\mH}(e) = nd - O(n^3)$, proving (W1). Since $w_v$ is a $1$-uniform test function, (W2)-(W4) are trivially satisfied. Therefore, for each $v \in V(K_n)$, $w_v$ is $(d, \eps, \mC)$-trackable.

We could now apply Theorem~\ref{thm:blackbox} to obtain properties (\ref{property1})-(\ref{property4}) of Theorem~\ref{thm:coloringproperties}. Indeed, for suitable $\eps\in(0,\frac{1}{4})$ and sufficiently large $n$, Theorem~\ref{thm:blackbox} yields a $\mC$-free matching $M \subset \mH$ such that for every $v \in V(K_n)$, we have 
\[w_v(M) > (1 - d^{-\eps^3})d^{-1}w_v(\mH) > (1 - n^{-\delta})n\]
for $\delta<\eps^3\log_n(d)$. Thus, for every $v \in V(K_n)$, there are at most $n^{1 - \delta}$ edges in $K_n$ incident to $v$ that do not belong to a triangle selected by $M$. This proves property (\ref{property4}).

In order to guarantee property (\ref{property5}) of Theorem~\ref{thm:coloringproperties}, we  define several additional test functions. These test function will ensure that we do not create bad subgraphs when coloring $L=E(K_n)-F$ by the second, random coloring. 

We will need some additional terminology to describe these new test functions.  For a conflict $C\in \mC$, we will call a subset $C'\subset C$ a \emph{subconflict} of $C$. Given a subgraph $G\subset L$ and a subconflict $C'\subset C$, we say $G$ \emph{completes the conflict $C$ with $C'$} if $G$ completes a bad subgraph in $K_n$ with the subgraph of $F$ corresponding to $C'$. That is, an uncolored subgraph $G$ completes $C$ with $C'$ if there is a way to color the edges of $G$ so that $G$ together with the colored subgraph of $F$ corresponding to $C'$ is a bad subgraph. In particular, we will use this idea when $G$ is a single edge in $K_n$ or a pair of disjoint edges in $K_n$. 

For all distinct $x,y\in V(K_n)$, for $j_x,j_y\in\{1,2\}$, and for any type $t$ of conflict, we define $\mP_{j_x,j_y,t}$ to be the set of all subconflicts $C'$ with the following properties:
\begin{enumerate}
    \item $C'$ contains $\mH$-edges $e_x=(K_x,\alpha_x,\beta_x)$ and $e_y=(K_y,\alpha_y,\beta_y)$ containing $x$ and $y$, respectively, such that either $K_x\cap K_y=\emptyset$ and $|\{\alpha_x,\beta_x\}\cap\{\alpha_y,\beta_y\}|\geq 1$ or $|(K_x\cap K_y)-\{x,y\}|=1$ and $\{\alpha_x,\beta_x\}\cap\{\alpha_y,\beta_y\}=\emptyset$,
    \item there is an edge $x'y'\in E(L)$ disjoint from $xy$ such that $\{xy,x'y'\}$ completes a conflict $C$ of type $t$ with $C'$, and \label{subconflict}
    \item for each $z\in\{x,y\}$, if $\alpha_z$ is the color incident to $z$ in $K_z$ which appears in the bad subgraph corresponding to $C$, then $\alpha_z$ is incident to $z$ exactly $j_z$ times in $K_z$. In the special case $t=c$ where all three edges of $K_x$ appear in the bad subgraph, set $j_x=2$. 
\end{enumerate}

Note that given a conflict of type $t$ with size $j\in\{4,5,6\}$, the subconflicts in $\mP_{j_x,j_y,t}$ will have size $j-2$. Furthermore, some of these sets will be empty, as there may be no subconflicts for a particular choice of $x,y,j_x,$ and $j_y$, so we disregard these cases for the rest of the proof. 

We can show using McDiarmid's Inequality that in the random 9-uniform hypergraph $\mH$ considered earlier, we have for all distinct $x,y\in V(K_n)$, $j_x,j_y\in\{1,2\}$, and $t\in\{a,b,c\}$ that with high probability, \[|\mP_{j_x,j_y,t}|=\frac{\left(p(1-p)^5\right)^2}{j_xj_y}\cdot k^3n^4\pm O(n^{20/3}).\]
Indeed, for $w\in V\backslash\{x,y\}$, define $b_w=n^6$ if $w$ is a copy of $x$ or $y$ and $b_w=n^5$ otherwise; hence $\sum_{w\in V} b_w^2=O(n^{13})$.
We can similarly show for all $x,y\in V(K_n)$, $j_x,j_y\in\{1,2\}$, and $t\in\{d,e,f\}$ that with high probability, if $j$ is the size of a conflict of type $t$, then 
\[|\mP_{j_x,j_y,t}|=\frac{\left(p(1-p)^5\right)^{j-2}}{j_xj_y}\cdot k^jn^{2j-5}\pm O(n^{3j-\frac{16}{3}}).\] 
Since $k=n$, this gives for all $t\in\{a,b,c,d,e,f\}$ that with high probability, 
\[|\mP_{j_x,j_y,t}|=\left(\frac{5^5}{6^6}\right)^{j-2}\frac{n^{3j-5}}{j_xj_y}\pm O(n^{3j-\frac{16}{3}}).\]
We will assume from now on that we have chosen $\mH$ such that this property holds. 

Let $w_{x,y,j_x,j_y,t}$ be the indicator weight function for the subconflicts in $\mP_{j_x,j_y,t}$. Assume for now that these are $(d,\eps,\mC)$-trackable test functions for $\mH$ for all $\eps\in (0,\frac{1}{4})$. 
By including these weight functions when we apply Theorem~\ref{thm:blackbox}, we obtain a matching $M$ such that 
\begin{equation}\label{eq:pairs}
\left|\binom{M}{j-2}\cap \mP_{j_x,j_y,t}\right|=w_{x,y,j_x,j_y,t}(M)\leq  (1+d^{-\eps^3})d^{-j+2}|\mP_{j_x,j_y,t}|\leq(1+n^{-\delta})\frac{n}{3^{j-2}j_xj_y}\end{equation}
for each $x,y,j_x,j_y,t$.

In addition, we define $\mT_{j_x,j_y,t}$ to be the set of all subconflicts $C$ with the same properties as $\mP_{j_x,j_y,t}$, except that property~\ref{subconflict} is replaced by the condition that $\{xy\}$ completes a conflict of type $t$ with $C$. So, we can think of each subconflict $C$ in $\mT_{j_x,j_y,t}$ as extending a subconflict $C'$ in $\mP_{j_x,j_y,t}$ by one $\mH$-edge $(K,\gamma,\gamma')$ where $K$ is edge-disjoint from $K_x$ and $K_y$ and contains the graph edge $x'y'$ with which $xy$ completes $C'$. Thus, given a conflict of type $t$ with size $j\in\{4,5,6\}$, the subconflicts in $\mT_{j_x,j_y,t}$ will have size $j-1$. As before, some of these sets will be empty, as there may be no subconflicts for a particular choice of $x,y,j_x,$ and $j_y$, so we disregard these cases for the rest of the proof. 

Fix some $C'\in \mP_{j_x,j_y,t}$. Since $\mH$ is essentially $d$-regular, we have $d_{\mH}(x'y') = d \pm O(n^2)$, and since $\Delta_2(\mH)=O(d/n)$, almost all edges containing $x'y'$ in $\mH$ form a matching with $e_x$ and $e_y$. Thus, we have
\[|\mT_{j_x,j_y,t}|=j_xj_y(d\pm O(n^2))|\mP_{j_x,j_y,t}|.\]

Let $w'_{x,y,j_x,j_y,t}$ be the indicator weight function for the subconflicts in $\mT_{j_x,j_y,t}$, and again assume for now that these are $(d,\eps,\mC)$-trackable test functions for $\mH$ for all $\eps\in (0,\frac{1}{4})$. Applying Theorem~\ref{thm:blackbox} with all of our weight functions, we obtain a matching $M$ such that 
\begin{equation}\label{eq:triples}
\left|\binom{M}{j-1}\cap \mT_{j_x,j_y,t}\right|=w'_{x,y,j_x,j_y,t}(M)\geq(1-d^{-\eps^3})d^{-(j-1)}|\mT_{j_x,j_y,t}|\geq(1-n^{-\delta})\frac{n}{3^{j-2}}\end{equation}
for each $x,y,j_x,j_y,t$. 
 
By (\ref{eq:pairs}) and (\ref{eq:triples}), the number of edges described in property~(\ref{property5}) of Theorem~\ref{thm:coloringproperties} is at most 
\[\sum_{j_x,j_y\in \{1,2\}} j_xj_y\left|\binom{M}{j-2}\cap \mP_{j_x,j_y,t}\right|-\left|\binom{M}{j-1}\cap \mT_{j_x,j_y,t}\right|\leq n^{1-\delta}.\]
This  proves (\ref{property5}). 

Thus, all that remains is to show that $w_{x,y,j_x,j_y,t}$ and $w'_{x,y,j_x,j_y,t}$ are $(d,\eps,\mC)$-trackable test functions for $\mH$ for all $\eps\in(0,\frac{1}{4})$. 
By our estimations of $|\mP_{j_x,j_y,t}|=\Theta(nd^{j-2})$ and $|\mT_{j_x,j_y,t}|=\Theta(nd^{j-1})$, condition (W1) holds for $w_{x,y,j_x,j_y,t}$ and $w'_{x,y,j_x,j_y,t}$. Also note that condition (W4) is vacuously true for both functions. 

To see condition (W2) for $w_{x,y,j_x,j_y,t}$, fix an edge $e=(K,i,\ell)$ in $\mH$ with $K=xuv$, and suppose that $e$ is in at least one subconflict $C'\in\mP_{j_x,j_y,t}$. We consider cases based on the size $j$ of the conflicts of type $t$. First, let $j=4$.  Then any edge $f$ for which $C'=\{e,f\}$ must contain either one of the $\mH$-vertices $y_i$ or $y_{\ell}$ or one of the graph edges $\{yu,yv\}$. Thus, there are $O(d)<n^7/d^{1+\eps}$ pairs in $\mP_{j_x,j_y,t}$ containing $e$. 

If instead $j=5$, then there are two inequalities to check. Note that any subconflict $C'=\{e,f,f'\}$ in $\mP_{j_x,j_y,t}$ must have an edge $f$ containing an $\mH$-vertex in $\{y_i,y_{\ell},yu,yv\}$, and an edge $f'$ containing either two graph vertices in $e\cup f$ (and hence an $\mH$-vertex of the type $ab$) or a graph vertex and a color in $e\cup f$ (and hence an $\mH$-vertex of the type $a_i$). Thus, there are at most $O(d^2)<n^{10}/d^{1+\eps}$ such subconflicts containing $e$. If we now also fix a second edge $f=(K',s,t)$ in $\mH$ which is in at least one subconflict $C'$ with $e$ in $\mP_{j_x,j_y,t}$, then either $s\in\{i,\ell\}$ (or $t\in \{i,\ell\}$) or $K$ and $K'$ share a vertex. In the first case, the third edge $f'$ in any $C'$ containing $e,f$ must contain one of the graph edges between a vertex in $K$ and one in $K'$, and in the second case, $f'$ must contain a graph vertex from $K$ or $K'$ and a color in $\{i,\ell,s,t\}$.  So, there are $O(d)<n^{10}/d^{2+\eps}=n^{4-\eps}$ such subconflicts containing $e,f$. 
 
The cases for $j=6$ are similar. There are $O(d^3)<n^{13}/d^{1+\eps}$ subconflicts $C'=\{e,e',f,f'\}\in \mP_{j_x,j_y,t}$ which contain $e$ since there are $k^2n^2$ ways to pick an edge $e'$ which shares a graph vertex with $e$, then $O(d)$ ways to pick a second edge $f$ which shares a graph vertex with $e\cup e'$ with a fixed color, and finally $O(d/n)$ ways to pick the third edge $f'$ since we know both a graph edge and color it must contain. In addition, there are $O(d^2)<n^{13}/d^{2+\eps}$ ways to pick a subconflict containing a fixed pair $e,e'$ and $O(d)<n^{13}/d^{3+\eps}$ ways to pick a subconflict containing a fixed triple $e,e',f$. Thus, condition (W2) is satisfied for $w_{x,y,j_x,j_y,t}$. 

To see property (W2) for $w'_{x,y,j_x,j_y,t}$, recall that each subconflict $C$ in $\mT_{j_x,j_y,t}$ is formed by adding an edge $f$ in $\mH$ to a subconflict $C'$ in $\mP_{j_x,j_y,t}$, where $f$ must contain the graph edge $x'y'$ with which $xy$ completes $C'$. So, there are at most $O(d)$ subconflicts $C$ which extend a particular $C'$. Now suppose $t$ has size $j$, fix an edge $e=(K,i,\ell)$ in $\mH$, and suppose that $e$ is in at least one subconflict $C\in\mT_{j_x,j_y,t}$. Since $w_{x,y,j_x,j_y,t}$ satisfies property (W2), the number of subconflicts $C$ containing $e$ is at most $O(d)\cdot n^{3j-5}/d^{1+\eps}\leq w'_{x,y,j_x,j_y,t}/d^{1+\eps}$, as desired. The cases where we fix two or three edges in $\mH$ follow similarly. 

Now we will show that $w_{x,y,j_x,j_y,t}$ satisfies property (W3). To this end, fix two $\mH$-edges $e=(K,i,\ell)$ and $f=(K',s,t)$ which are in at least one subconflict in $\mP_{j_x,j_y,t}$ together. We will show for each $j\in\{4,5,6\}$ that $|(\mC_e)^{j-1}\cap (\mC_f)^{j-1}|\leq d^{j-1-\eps}$. 
Note that the number of conflicts of size $j\in\{4,5,6\}$ containing $e$ (and not necessarily $f$) is at most $\Delta(C^{(j)})=O(d^{j-1})$. 
However, any subconflict of size $j-1$ which also  completes a conflict with $f$ must use either two additional fixed vertices (if $K$ and $K'$ are disjoint) or one other fixed vertex and one more fixed color (if $K$ and $K'$ intersect in a vertex), and thus there are at most $O(d^{j-1}/n^2)<O(d^{j-1-\eps})$ such subconflicts.
By the same reasoning, $w'_{x,y,j_x,j_y,t}$ satisfies property (W3). 
\qed

\subsection{Proof of Theorem~\ref{thm:upperbound}}

Applying Theorem~\ref{thm:coloringproperties} with $2\delta$ in place of $\delta$, we obtain a coloring of a subgraph $F\subset K_n$ with the five desired properties. In particular, the remaining uncolored subgraph $L=K_n-E(F)$ has maximum degree $\Delta(L)\leq n^{1-\delta}$ by property (\ref{property4}). We now randomly color the edges of $L$ from a set $P$ of $k=n^{1-\delta}$ new colors. For each edge in $L$, we assign its color with equal probability $1/k$, independently of the other edges. 

We will show using the Local Lemma that the union of these two colorings of $F$ and $L$ is a $(5,8)$-coloring of $K_n$. In order to do so, we define several types of bad events. First, for any pair of adjacent edges $e,f$ in $L$, we define $A_{e,f}$ to be the event that both $e$ and $f$ receive the same color. Then $\prob(A_{e,f})=k^{-1}$. 

The other bad events will correspond to appearances of bad subgraphs. By our construction of the coloring of $F$, none of these subgraphs can appear in $K_n$ using only edges colored in $F$. Furthermore, since we use disjoint sets of colors on $F$ and $L$, and each color in a bad subgraph appears twice (except for type $b$ bad subgraphs, where one color appears three times), it suffices to define at most three types of bad events for each type of bad subgraph (with one, two, or three monochromatic matchings coming from $L$). Some of the bad subgraphs do not require all three types of bad events; for example, in  bad subgraphs of type $c$, the 2-colored triangle must be in $F$, so a single type of bad event suffices in this case. 

We will say that a subgraph $H$ of $K_n$ is \emph{potentially bad} if there is a way to color its edges in $L$ using colors from $P$ that would create a bad subgraph. That is, $G$ is potentially bad if $H\cap L$ completes $H$ into a bad subgraph. For example, any copy of $C_4$ in $L$ is potentially bad, as is any copy of $C_4$ in which one pair of matching edges is in $F$ and the other is in $L$. 

For each potentially bad subgraph $H$ in $K_n$ and corresponding type $t$ of bad subgraph, let $B_{H,t}$ be the event that the edges of $H\cap L$ receive colors from $P$ which make $H$ into a bad subgraph of type $t$. Note that if $m=|H\cap L|$, then $m\in\{2,4,6\}$ since $H \cap L$ can consists of one, two, or three 2-edge matchings.  
Then $\prob(B_{H,t})\leq 2k^{-m/2}$.

Let $\mE$ be the set of all bad events defined above. Two events are edge-disjoint if their corresponding edges in $L$ are distinct. Let $E\in\mE$. There are at most 6 ways to pick a graph edge $xy$ in $E$, and for each type $t$ of bad event, we will bound the number of these which share the edge $xy$ with $E$. There are at most $\Delta(L)=kn^{-\delta}$ events $A_{e,f}$ which contain $xy$. Now we will consider events $B_{H,t}$. If $m=2$, then by property (\ref{property5}) of the matching used to color $F$, we know there are at most $n^{1-2\delta}=kn^{-\delta}$ events $B_{H,t}$ which contain $xy$. 

If $m=4$, then $t\in\{a,e,f\}$.  First, suppose $t = a$.  There are $O(\Delta(L))^2=O(k^2n^{-2\delta})$ events $B_{H,a}$ which contain $xy$, since there are $\Delta(L)$ ways each to pick a neighbor of $x$ and a neighbor of $y$ in $L$ to complete the 4-cycle. Now suppose $t \in \{e,f\}$. For any bad subgraph $H$ of type $t$, $H \cap L$ must be a path on 5 vertices with both endpoints in $H\cap F$. 
There are $O(\Delta(L)^2)=O(k^2n^{-2\delta})$ ways to pick a 4-vertex path containing $xy$ in $L$. Note that the fifth vertex of $H$ is determined by this choice of path. Indeed,  some edge $f$ of $H\cap F$ must be induced by the four vertices on the path, and this edge comes from an $\mH$-edge which either contributes a monochromatic 2-edge path to $H$ or contributes one edge of a monochromatic 2-edge matching to $H$. In either case, the edge $f'$  in $H\cap F$ which receives the same color as $f$  determines the fifth vertex of $H$, and hence, fixes the bad rest of the bad subgraph. Thus, there are $O(\Delta(L)^2)=O(k^2n^{-2\delta})$ events $B_{H,t}$ with $m=4$ containing $xy$.  

Finally, if $m=6$, then it must be the case that $t=f$. There are at most $O(\Delta(L)^3)$ ways to create a 5-vertex path in $L$ containing $xy$, and hence, to create a potentially bad subgraph of type $f$. Thus, there are at most $O(\Delta(L)^3) = O(k^3 n^{-3\delta})$ bad events $B_{H,f}$ containing $xy$.

To apply the Local Lemma, we now assign a number $x_E\in [0,1)$ to each bad event $E\in \mE$. For each bad event of type $A_{e,f}$, let $x_A=10/k$. For each bad event of type $B_{H,t}$ with $m=|H\cap L|$, let $x_{B,m}=10/k^{m/2}$. Note that the probability of any event $A_{e,f}$ is $k^{-1}$, which is smaller than 
\[x_A(1-x_A)^{O(kn^{-2\delta})}(1-x_{B,2})^{O(kn^{-2\delta})}(1-x_{B,4})^{O(k^2n^{-2\delta})}(1-x_{B,6})^{O(k^3n^{-2\delta})}=(1+o(1))x_A.\]
In addition, for each $m\in\{2,4,6\}$, the  probability of any event $B_{H,t}$ with $m=|H\cap L|$ is at most $2k^{-m/2}$, which is smaller than 
\[x_{B,m}(1-x_A)^{O(kn^{-2\delta})}(1-x_{B,2})^{O(kn^{-2\delta})}(1-x_{B,4})^{O(k^2n^{-2\delta})}(1-x_{B,6})^{O(k^3n^{-2\delta})}=(1+o(1))x_{B,m}.\]
Thus, condition~(\ref{LLL}) holds, and Lemma~\ref{lem:LLL} implies that with positive probability, our colorings of $F$ and $L$ give a $(5,8)$-coloring of $K_n$, as desired. 
\qed

\section*{Acknowledgements}
We thank the anonymous referees for their many useful comments and suggestions. Work on this project started during the Research Training Group (RTG) rotation at Iowa State University in the spring of 2023. Emily Heath, Alex Parker, and Coy Schwieder  were supported by NSF grant DMS-1839918. 
Shira Zerbib was supported by NSF grant DMS-1953929.
We would like to thank Chris Wells, Bernard Lidick\'y and Ryan Martin for fruitful discussions during early stages of this project. 

\bibliography{bibfile}

\begin{thebibliography}{10}

\bibitem{AS}
N.~Alon and J.~H. Spencer.
\newblock {\em The probabilistic method}.
\newblock Wiley Series in Discrete Mathematics and Optimization. John Wiley \& Sons, Inc., Hoboken, NJ, fourth edition, 2016.

\bibitem{axenovich2000}
M.~Axenovich.
\newblock A generalized {R}amsey problem.
\newblock {\em Discrete Mathematics}, 222(1-3):247--249, 2000.

\bibitem{AFM}
M.~Axenovich, Z.~F\"uredi, and D.~Mubayi.
\newblock On generalized {R}amsey theory: The bipartite case.
\newblock {\em Journal of Combinatorial Theory, Series B}, 79(1):66--86, 2000.

\bibitem{BEHK}
J.~Balogh, S.~English, E.~Heath, and R.~A. Krueger.
\newblock Lower bounds on the {E}rd{\H{o}}s--{G}y{\'a}rf{\'a}s problem via color energy graphs.
\newblock {\em Journal of Graph Theory}, 103(2):378--409, 2023.

\bibitem{BCDP}
P.~Bennett, R.~Cushman, A.~Dudek, and P.~Pra{\l}at.
\newblock The {E}rd{\H{o}}s--{G}y\'arf\'as function $f(n,4,5)=\frac{5}{6}n+o(n)$ -- so {G}y\'arf\'as was right.
\newblock {\em arXiv:2207.02920}, 2022.

\bibitem{BDLP}
P.~Bennett, M.~Delcourt, L.~Li, and L.~Postle.
\newblock On generalized {R}amsey numbers in the sublinear regime.
\newblock {\em arXiv:2212.10542}, 2022.

\bibitem{BDE}
P.~Bennett, A.~Dudek, and S.~English.
\newblock A random coloring process gives improved bounds for the {E}rd{\H{o}}s--{G}y\'arf\'as problem on generalized {R}amsey numbers.
\newblock {\em arXiv:2212.06957}, 2022.

\bibitem{BHZ}
P.~Bennett, E.~Heath, and S.~Zerbib.
\newblock Edge-coloring a graph {$G$} so that every copy of a graph {$H$} has an odd color class.
\newblock {\em arXiv:2307.01314}, 2023.

\bibitem{56}
A.~Cameron.
\newblock An explicit edge-coloring of $k_n$ with six colors on every {$K_5$}.
\newblock {\em Electronic Journal of Combinatorics}, 26:4, 2017.

\bibitem{CH1}
A.~Cameron and E.~Heath.
\newblock A $(5,5)$-colouring of ${K}_n$ with few colours.
\newblock {\em Combinatorics, Probability and Computing}, 27(6):892--912, 2018.

\bibitem{CH2}
A.~Cameron and E.~Heath.
\newblock New upper bounds for the {E}rd{\H{o}}s--{G}y\'arf\'as problem on generalized {R}amsey numbers.
\newblock {\em Combinatorics, Probability and Computing}, 32(2):349–362, 2023.

\bibitem{CFLS}
D.~Conlon, J.~Fox, C.~Lee, and B.~Sudakov.
\newblock The {E}rd{\H{o}}s--{G}y{\'a}rf{\'a}s problem on generalized {R}amsey numbers.
\newblock {\em Proceedings of the London Mathematical Society}, 110(1):1--18, 2015.

\bibitem{DP}
M.~Delcourt and L.~Postle.
\newblock Finding an almost perfect matching in a hypergraph avoiding forbidden submatchings.
\newblock {\em arXiv:2204.08981}, 2022.

\bibitem{DP2}
M.~Delcourt and L.~Postle.
\newblock The limit in the $(k+ 2, k) $-problem of {B}rown, {E}rd{\H{o}}s and {S}\'os exists for all $k\geq 2$.
\newblock {\em arXiv:2210.01105}, 2022.

\bibitem{erdos1975}
P.~Erd\H{o}s.
\newblock Problems and results on finite and infinite graphs.
\newblock {\em Recent Advances in Graph Theory (Proc. Second Czechoslovak Sympos., Prague, 1974)}, pages 183--192, 1975.

\bibitem{erdos1981}
P.~Erd{\H{o}}s.
\newblock Solved and unsolved problems in combinatorics and combinatorial number theory.
\newblock {\em European Journal of Combinatorics}, 2:1--11, 1981.

\bibitem{EG}
P.~Erd{\H{o}}s and A.~Gy{\'a}rf{\'a}s.
\newblock A variant of the classical {R}amsey problem.
\newblock {\em Combinatorica}, 17(4):459--467, 1997.

\bibitem{FPS}
S.~Fish, C.~Pohoata, and A.~Sheffer.
\newblock Local properties via color energy graphs and forbidden configurations.
\newblock {\em SIAM Journal on Discrete Mathematics}, 34(1):177--187, 2020.

\bibitem{GJKKL}
S.~Glock, F.~Joos, J.~Kim, M.~K{\"u}hn, and L.~Lichev.
\newblock Conflict-free hypergraph matchings.
\newblock In {\em Proceedings of the 2023 Annual ACM-SIAM Symposium on Discrete Algorithms (SODA)}, pages 2991--3005. SIAM, 2023.

\bibitem{GJKKLP}
S.~Glock, F.~Joos, J.~Kim, M.~K{\"u}hn, L.~Lichev, and O.~Pikhurko.
\newblock On the $(6, 4)$-problem of {B}rown, {E}rd{\H{o}}s and {S}\'os.
\newblock {\em arXiv:2209.14177}, 2022.

\bibitem{JM}
F.~Joos and D.~Mubayi.
\newblock Ramsey theory constructions from hypergraph matchings.
\newblock {\em arXiv:2208.12563}, 2022.

\bibitem{mcdiarmid}
C.~McDiarmid et~al.
\newblock On the method of bounded differences.
\newblock {\em Surveys in combinatorics}, 141(1):148--188, 1989.

\bibitem{Mubayi1}
D.~Mubayi.
\newblock Edge-coloring cliques with three colors on all 4-cliques.
\newblock {\em Combinatorica}, 18(2):293--296, 1998.

\bibitem{mubayi2004}
D.~Mubayi.
\newblock An explicit construction for a {R}amsey problem.
\newblock {\em Combinatorica}, 24(2):313--324, 2004.

\bibitem{PS}
C.~Pohoata and A.~Sheffer.
\newblock Local properties in colored graphs, distinct distances, and difference sets.
\newblock {\em Combinatorica}, 39(3):705--714, 2019.

\bibitem{sarkozy2000edge}
G.~N. S{\'a}rk{\"o}zy and S.~Selkow.
\newblock On edge colorings with at least $q$ colors in every subset of $p$ vertices.
\newblock {\em The Electronic Journal of Combinatorics}, 8(1):R9, 2000.

\bibitem{sarkozy2003application}
G.~N. S{\'a}rk{\"o}zy and S.~M. Selkow.
\newblock An application of the regularity lemma in generalized {R}amsey theory.
\newblock {\em Journal of Graph Theory}, 44(1):39--49, 2003.

\end{thebibliography}
\bibliographystyle{abbrv}

\end{document}